\documentclass{ifacconf}

\usepackage{natbib}            
\usepackage{graphicx,amssymb,subfigure,url, amsmath}
\usepackage{afterpage}
\usepackage{latexsym}
\usepackage{bm}

\newcommand{\Real}{\mathbb R}
\newcommand{\Natural}{\mathbb N}
\newcommand{\Expected}[1]{\mathbf{E}\left[#1\right]}
\newcommand{\Trace}[1]{\mathbf{Tr}\left[#1\right]}

\renewcommand{\v}[1]{\mathbf{#1}}

\newtheorem{proposition}{Proposition}

\newtheorem{definition}{Definition}

\newcommand{\eqnref}[1]{eqn.(\ref{#1})}

\newcommand{\Fig}[1]{Figure (\ref{#1})}

\newenvironment{proof}[1][Proof]{\begin{trivlist}
\item[\hskip \labelsep {\bfseries #1}]}{\end{trivlist}}

\begin{document}

\begin{frontmatter}

\title{Linear Receding Horizon Control with Probabilistic System Parameters} 

\author[First]{Raktim Bhattacharya} 
\author[Second]{James Fisher} 

\address[First]{Aerospace Engineering, Texas A\&M University, USA.}                                              
\address[Second]{Raytheon Missile Systems.}


\begin{abstract}                          
In this paper we address the problem of designing receding horizon control algorithms for linear discrete-time systems with parametric uncertainty. We do not consider presence of stochastic forcing or process noise in the system. It is assumed that the parametric uncertainty is probabilistic in nature with known probability density functions. We use generalized polynomial chaos theory to design the proposed stochastic receding horizon control algorithms. In this framework, the stochastic problem is converted to a deterministic problem in higher dimensional space. The performance of the proposed receding horizon control algorithms is assessed using a linear model with two states. 
\end{abstract}

\end{frontmatter}

\section{Introduction}
Receding horizon control (RHC), also known as model predictive control (MPC), has been popular in the process control industry for several years \cite{MPC_Overview,MPCSurvey}, and recently gaining popularity in aerospace applications, see  \cite{raktimf16}. It is based on the idea of repetitive solution of an optimal control problem and updating states with the first input of the optimal command sequence. The repetitive nature of the algorithm results in a state dependent feedback control law. The attractive aspect of this method is the ability to incorporate state and control limits as constraints in the optimization formulation. When the model is linear, the optimization problem is quadratic if the performance index is expressed via a $\mathcal{L}_2 $-norm, or linear if expressed via a $\mathcal{L}_1/\mathcal{L}_\infty$-norm. Issues regarding feasibility of online computation, stability and performance are largely understood for linear systems and can be found in refs. \cite{kwon94,rhcbook1}. For nonlinear systems, stability of RHC methods is guaranteed by \cite{Primbs-Thesis,Ali_ACC99}, by using an appropriate control Lyapunov function . For a survey of the state-of-the-art in nonlinear receding horizon control problems the reader is directed to \cite{NLRHC_survey}.

Traditional RHC laws perform best when modeling error is small. \cite{myrhc} has shown that system uncertainty can lead to significant oscillatory behavior and possibly instabilty. Furthermore, \cite{rhcnonrobust} showed that in the presence of modeling uncertainty RHC strategy may not be robust with RHC designs. Many approaches have been taken to improve robustness of RHC strategy in the presence of unknown disturbances and bounded uncertainty, see work of \cite{rakovic, rhcworstcase, rhcefficient, mayne2000cmp}. These approaches involve the computation of a feedback gain to ensure robustness. The difficulty with this approach is that, even for linear systems, the problem becomes difficult to solve, as the unknown feedback gain transforms the quadratic programming problem into a nonlinear programming problem. 

In this paper we address the problem of RHC design for linear systems with probabilistic uncertainty in system parameters. Parametric uncertainty arises in systems when the physics governing the system is known and the system parameters are either not known precisely or are expected to vary in the operational lifetime. Such uncertainty also occurs when system models are build from experimental data using system identification techniques. As a result of experimental measurements, the values of the parameters in the system model have a range of variations with quantifiable likelihood of occurrence.  In either case, the range of variation of these parameters and the likelihood of their occurrence are assumed to be known and it is desired to design controllers that achieve specified performance for these variations.

While the area of robust RHC is not new, approaching the problem from a stochastic standpoint is only recently receiving attention, for example \cite{hessem1,Batina}. These approaches however suffered from  either computational complexity, high degree of conservativeness or do not address closed-loop stability. The key difficulty in stochastic RHC is the propagation of uncertainty over the prediction horizon. More recently, \cite{Cannon2009167} avoid this difficulty by using an autonomous augmented formulation of the prediction dynamics. Constraint satisfaction and stability is achieved in \cite{Cannon2009167} by extending ellipsoid invariance theory to invariance with a given probability. The cost function minimized was the expected value of a quadratic function of random state and control trajectories. Additionally, the uncertainty in the system parameters were assumed to have normal distribution. 

This paper presents formulation of robust RHC design problems in polynomial chaos framework, where parametric uncertainty can be governed by any probability density function. In this approach the solution, not the dynamics, of the random process is approximated using a series expansion. It is assumed that the random process to be controlled has finite second moment, which is the assumption of the polynomial chaos framework. The polynomial chaos based approach predicts the propagation of uncertainty more accurately, is computationally cheaper than methods based on Monte-Carlo or series approximation of the dynamics, and is less conservative than the invariance based methods. 

The paper is organized as follows. We first present a brief introduction to polynomial chaos and its application in transforming linear stochastic dynamics to linear deterministic dynamics in higher dimensional state-space. Next stability of stochastic linear dynamics in the polynomial chaos framework is presented. This is followed by formulation of RHC design for discrete-time stochastic linear systems. Stability of the proposed RHC algorithm is then analyzed. The paper concludes with numerical examples that assesses the performance of the proposed method. 

\section{Background on Polynomial Chaos}
Recently, use of polynomial chaos to study stochastic differential equations is gaining popularity. It is a non-sampling based method to determine evolution of uncertainty in dynamical system, when there is probabilistic uncertainty in the system parameters. Polynomial chaos was first introduced by \cite{wienerPC}
where Hermite polynomials were used to model stochastic processes
with Gaussian random variables. It can be thought of as an extension of Volterra's theory of nonlinear functionals \cite{volterra} for stochastic systems \cite{pcFEM}. According to \cite{CameronMartin} such an expansion converges in the
$\mathcal{L}_2$ sense for any arbitrary stochastic process with
finite second moment. This applies to most physical systems. \cite{Xiu} generalized the result of Cameron-Martin to various
continuous and discrete distributions using orthogonal polynomials
from the so called Askey-scheme \cite{Askey-Polynomials} and
demonstrated $\mathcal{L}_2$ convergence in the corresponding Hilbert
functional space. This is popularly known as the generalized
polynomial chaos (gPC) framework. The gPC framework has been applied to applications including stochastic fluid dynamics \cite{pcFluids2},stochastic finite elements \cite{pcFEM}, and solid mechanics
\cite{pcSolids1}. It has been shown in \cite{Xiu} that gPC based methods are computationally far superior than Monte-Carlo based methods. However, application of gPC to control related problems has been surprisingly limited and is only recently gaining popularity. See \cite{vinh-JGCD, fisher2008sld, pctrajgen} for control related application of gPC theory.

\subsection{Wiener-Askey Polynomial Chaos}
Let $(\Omega,\mathcal{F},P)$ be a probability space, where $\Omega$
is the sample space, $\mathcal{F}$ is the $\sigma$-algebra of the
subsets of $\Omega$, and $P$ is the probability measure. Let
$\Delta(\omega) =
(\Delta_1(\omega),\cdots,\Delta_d(\omega)):(\Omega,\mathcal{F})\rightarrow(\Real^d,\mathcal{B}^d)$
be an $\Real^d$-valued continuous random variable, where
$d\in\Natural$, and $\mathcal{B}^d$ is the $\sigma$-algebra of Borel
subsets of $\Real^d$. A general second order process $X(\omega)\in
\mathcal{L}_2(\Omega,\mathcal{F},P)$ can be expressed by polynomial
chaos as
\begin{equation}
\label{eqn.gPC}
X(\omega) = \sum_{i=0}^{\infty} x_i\phi_i({\Delta}(\omega)),
\end{equation}
where $\omega$ is the random event and $\phi_i({\Delta}(\omega))$
denotes the gPC basis of degree $p$ in terms of the random variables
$\Delta(\omega)$. The functions $\{\phi_i\}$ are a family of
orthogonal basis in $\mathcal{L}_2(\Omega,\mathcal{F},P)$ satisfying
the relation
\begin{equation}
\langle \phi_i \phi_j \rangle := \int_{\mathcal{D}_{\Delta(\omega)}}{\phi_i\phi_jw(\Delta(\omega))
\,d\Delta(\omega)}=h_i^2\delta_{ij}
\end{equation}
where $\delta_{ij}$ is the Kronecker delta, $h_i$ is a constant
term corresponding to $\int_{\mathcal{D}_{\Delta}}{\phi_i^2w(\Delta)\,d\Delta}$,
$\mathcal{D}_{\Delta}$ is the domain of the random variable $\Delta(\omega)$, and
$w(\Delta)$ is a weighting function. Henceforth, we will use $\Delta$ to
represent $\Delta(\omega)$.
For random variables $\Delta$ with certain distributions, the family
of orthogonal basis functions $\{\phi_i\}$ can be chosen in such a
way that its weight function has the same form as the probability
density function $f(\Delta)$.  When these types of polynomials are chosen,
we have $f(\Delta)=w(\Delta)$ and
\begin{equation}
\int_{\mathcal{D}_{\Delta}}{\phi_i\phi_jf(\Delta)\,d\Delta}=
\Expected{\phi_i\phi_j} = \Expected{\phi_i^2}\delta_{ij},
\end{equation}
where $\Expected{\cdot}$
denotes the expectation with respect to the probability measure
$dP(\Delta(\omega))=f(\Delta(\omega))d\Delta(\omega)$ and probability density
function $f(\Delta(\omega))$.
The orthogonal polynomials that are chosen are the
members of the Askey-scheme of polynomials (\cite{Askey-Polynomials}),
which forms a complete basis in the Hilbert space determined by
their corresponding support. Table \ref{table.pc} summarizes the
correspondence between the choice of polynomials for a given distribution
of $\Delta$. See \cite{Xiu} for more details.

\begin{table}[htbp]
\centering
\begin{tabular}{|c|c|}
\hline
Random Variable $\Delta$ & $\phi_i(\Delta)$ of the Wiener-Askey Scheme\\ \hline
Gaussian & Hermite \\
Uniform  & Legendre \\
Gamma   & Laguerre \\
Beta    & Jacobi\\\hline
\end{tabular}
\caption{Correspondence between choice of polynomials and given
distribution of $\Delta(\omega)$ \cite{Xiu}.} \label{table.pc}
\end{table}

\subsection{Approximation of Stochastic Linear Dynamics Using Polynomial Chaos Expansions}
Here we derive a generalized representation of the deterministic dynamics obtained from the stochastic system by approximating the solution with polynomial chaos expansions. 

Define a linear discete-time stochastic system in the following manner
\begin{equation}
x(k+1,\Delta)=A(\Delta)x(k,\Delta) + B(\Delta)u(k,\Delta), \label{eqn:lti}
\end{equation}
where $x\in\Real^n, u\in\Real^m$. The system has probabilistic
uncertainty in the system parameters, characterized by $A(\Delta),
B(\Delta)$, which are matrix functions of random variable
$\Delta\equiv\Delta(\omega)\in\Real^d$ with certain
\textit{stationary} distributions. Due to the stochastic nature of
$(A,B)$, the system trajectory $x(k,\Delta)$ will also be stochastic.

By applying the Wiener-Askey gPC expansion of finite order to $x(k,\Delta),
A(\Delta)$ and $B(\Delta)$, we get the following approximations,
\begin{eqnarray}
\hat{x}(k,\Delta) &=& \sum_{i=0}^{p}x_i(k)\phi_i(\Delta),\, x_i(k)\in\Real^n  \label{eqn:gPC.x}\\
\hat{u}(k,\Delta) &=& \sum_{i=0}^{p} u_{i}(k)\phi_i(\Delta),\, u_i(k)\in\Real^m\label{eqn:gPC.u}\\
\hat{A}(\Delta) &=& \sum_{i=0}^{p}A_{i}\phi_i(\Delta),\, A_i = \frac{\langle A(\Delta),\phi_i(\Delta) \rangle}{\langle \phi_i(\Delta)^2 \rangle}\in\Real^{n\times n}\\
\hat{B}(\Delta) &=& \sum_{i=0}^{p}B_{i}\phi_i(\Delta), \, B_i = \frac{\langle B(\Delta),\phi_i(\Delta) \rangle}{\langle \phi_i(\Delta)^2 \rangle}\in\Real^{n\times m}.
\end{eqnarray}
The inner product or ensemble average $\langle\cdot,\cdot \rangle$, used in the above equations and in the rest of the paper, utilizes the weighting function associated with the assumed probability distribution, as listed in table~\ref{table.pc}.

The number of terms $p$ is determined by the dimension $d$ of $\Delta$ and the order $r$ of the orthogonal polynomials $\{\phi_k\}$, satisfying $p+1 = \frac{(d+r)!}{d!r!}$. The $n(p+1)$ time varying coefficients, $\{x_{i}(k)\}; k=0,\cdots,p$, are obtained by substituting the
approximated solution in the governing equation (eqn.(\ref{eqn:lti})) and conducting Galerkin projection on the basis functions $\{\phi_k\}_{k=0}^p$, to yield $n(p+1)$ \textit{deterministic} linear system of equations, which given by
\begin{equation}
\label{eq:pcrhcdyn}
\v{X}(k+1) = \v{A}\v{X}(k) + \v{B}\v{U}(k),
\end{equation}
where
\begin{eqnarray}
\v{X}(k) &=& [x_0(k)^T\; x_1(k)^T \;\cdots x_p(k)^T]^T, \label{eqn:Xdef}\\
\v{U}(k) &=& [u_0(k)^T\; u_1(k)^T \;\cdots u_p(k)^T]^T. \label{eqn:Udef}\\
\end{eqnarray}
Matrices $\mathbf{A}\in\Real^{n(p+1)\times n(p+1)}$ and $\mathbf{B}\in\Real^{n(p+1)\times m}$ are defined as
\begin{eqnarray}
\v{A}  &=&  (W\otimes I_n)^{-1}\left[
\begin{array}{c} 
H_A(E_0\otimes I_n)\\ \vdots \\ H_A(E_p\otimes I_n) 
\end{array}\right],\\
\v{B} &=& (W\otimes I_n)^{-1}\left[
\begin{array}{c} 
H_B(E_0\otimes I_m)\\ \vdots \\ H_B(E_p\otimes I_m) 
\end{array}\right],
\end{eqnarray}
where $H_A  =  \left[ A_0 \, \cdots \, A_p \right]$,
$H_B  =  \left[ B_0 \, \cdots \, B_p \right]$, $W = diag(\langle \phi_0^2 \rangle, \cdots, \langle \phi_p^2 \rangle )$, and  
\[E_i  = \left[\begin{array}{ccc} \langle \phi_i,\phi_0,\phi_0 \rangle & \cdots & \langle \phi_i,\phi_0,\phi_p \rangle \\ \vdots & & \vdots \\ \langle \phi_i,\phi_p,\phi_0 \rangle & \cdots & \langle \phi_i,\phi_p,\phi_p \rangle
\end{array}\right],
\]
with $I_n$ and $I_m$ as the identity matrix of dimension $n\times n$ and $m\times m$ respectively. It can be easily shown that $\Expected{x(k)} = x_0(k)$, or $\Expected{x(k)} = \left[I_n \; 0_{n\times np}\right]\v{X}(k).$

Therefore, transformation of a stochastic
linear system with $x\in\Real^n,u\in\Real^m$, with $p^{th}$ order
gPC expansion, results in a \textit{deterministic} linear system
with increased dimensionality equal to $n(p+1)$.

\section{Stochastic Receding Horizon Control}
Here we develop a RHC methodology for stochastic linear systems similar to that developed for deterministic systems, presented by \cite{goodwin2005cca}. Let $x(k,\Delta)$ be the solution of the system in \eqnref{eqn:lti} with control $u(k,\Delta)$. Consider the following optimal control problem defined by,
\begin{align}
\label{eq:rhc1}
&V_N^*= \min\,V_N(\{x(k+1,\Delta)\},\{u(k,\Delta)\}) \\ \nonumber
&{\rm subject\,to:}\\
&x(k+1,\Delta)=A(\Delta)x(k,\Delta)+B(\Delta)u(k,\Delta),\\
&\textrm{Initial Condition: }x(0,\Delta);\\
&\mu(u(k,\Delta))\in\mathbb{U}\subset \Real^{m},\\
&\mu(x(k,\Delta))\in\mathbb{X}\subset \Real^{n},\\
&\mu(x(N,\Delta))\in\mathbb{X}_f\subset \mathbb{X},
\end{align}
for $k=0,\cdots,N-1$; where $N$ is the horizon length, $\mathbb{U}$ and $\mathbb{X}$ are feasible sets for
$u(k,\Delta)$ and $x(k,\Delta)$ with respect to control and state constraints. $\mu(\cdot)$ represents moments based constraints on state and control.The set
$\mathbb{X}_f$ is a terminal constraint set. The cost function $V_N$ is given by

\begin{equation}
\begin{split}
& V_N = \sum_{k=1}^{N} \mathbf{E}\left[ x^T(k,\Delta)Qx(k,\Delta) + \right. \\
& \left. u^T(k-1,\Delta)Ru(k-1,\Delta) \right] + C_f(x(N),\Delta), 
\end{split}
\label{eqn:V_N}
\end{equation}

where $C_f(x(N),\Delta)$ is a terminal cost function, and $Q=Q^T>0$, $R=R^T>0$ are matrices with appropriate dimensions.

\subsection{Control Structure}
Here we consider the control structure,  
\begin{equation}
u(k,\Delta) =  \bar{u}(k)+K(k)\left(x(k,\Delta)-\Expected{x(k,\Delta)}\right), \label{eqn:rhclaw2}
\end{equation}
where $\bar{u}(k)$, and $K(k)$ are unknown \textit{deterministic} quantities. This is similar to that proposed by Primbs \etal  \cite{primbs2009stochastic} and enables us to regulate the mean trajectory using open loop control and deviations about the mean using a state-feedback control. 

In terms of gPC coefficients, the system dynamics in \eqnref{eqn:lti} with the first control structure is given by \eqnref{eq:pcrhcdyn}. The system dynamics in term of the gPC expansions, with the second control structure is given by 
\begin{equation}
\v{X}(k+1)=(\v{A}+\v{B} (\v{M} \otimes K(k)))\v{X}(k)+\v{B}\bar{U}(k), \label{eqn:gPCDyn1}
\end{equation}
where $\bar{U}(k)= [1\,\,0_{1\times p}]^T \otimes \bar{u}(k)$ and $\v{M} = \left[\begin{array}{cc} 0 & 0_{1\times p} \\ 0_{p\times 1} & I_{p\times p}
 \end{array}\right]$.
 
\subsection{Cost Functions}
Here we derive the cost function in \eqnref{eqn:V_N} is derived in terms of the gPC coefficients $\v{X}$ and $\v{U}$. For scalar $x$, the quantity $\Expected{x^2}$ in terms of its gPC expansions is given by
\begin{equation}
\Expected{x^2} = \sum_{i=0}^{p}\sum_{j=0}^{p} x_ix_j\int_{\mathcal{D}_\Delta}\phi_i \phi_j f d\Delta
 = \mathbf{x}^TW\mathbf{x},
\end{equation}
where $\mathcal{D}_\Delta$ is the domain of $\Delta$ , $x_i$ are the
gPC expansions of $x$, $f\equiv f(\Delta)$ is the probability
distribution of $\Delta$. Here we use the notation $\v{x}$ to represent the gPC state vector for scalar $x$. The expression
$\Expected{x^2}$ can be generalized for $x\in\Real^n$ where
$\Expected{x^T Q x}$ is given by
\begin{equation}
\Expected{x^T Q x} = \mathbf{X}^T(W \otimes Q)\mathbf{X}.
\end{equation}

The expression for the cost function in \eqnref{eqn:V_N}, in terms of gPC states and control is
\begin{equation}
\begin{split}
&V_N =  \sum_{k=0}^{N-1} [\v{X}^T(k)\bar{Q}\v{X}(k) + \\
& (\bar{U}^T(k)+\v{X}^T(k)(\v{M} \otimes K^T(k)))\bar{R} (\bar{U}(k) + \\
& (\v{M} \otimes K(k)))\v{X}(k))] + C_f(x(N),\Delta),
\end{split}
\end{equation}
where $\bar{Q}=W\otimes Q$ and $\bar{R}=W\otimes R$.

In deterministic RHC, the terminal cost is the cost-to-go from the terminal state to the origin by the local controller \cite{goodwin2005cca}. In the stochastic setting, a local controller can be synthesized using methods presented in our previous work \cite{fisher2008sld}. The cost-to-go from a given stochastic state variable $x(N,\Delta)$ can then be written as 
\begin{equation}
C_f(x(N),\Delta) = \v{X}^T(N)P\v{X}(N), \label{eqn:term-cost}
\end{equation}
where $\v{X}(N)$ are gPC states corresponding to $x(N,\Delta)$ and $P=P^T>0$ is a $n(p+1)\times n(p+1)$-dimensional matrix, obtained from the synthesis of the terminal control law \cite{fisher2008sld}. In the current stochastic RHC literature, the terminal cost function has been defined on the expected value of the final state \cite{lee1998ofc,delapenad2005spa,primbs2009stochastic,bertsekas2005dpa} or using a combination of mean and variance \cite{darlington2000decreasing,nagy2003robust}. The terminal cost function in \eqnref{eqn:term-cost} is more general than the terminal cost functions used in the literature because it penalizes all the moments of the random variable $x(N,\Delta)$, as they are functions of $\v{X}(N)$. This can be shown as follows.

To avoid tensor notation and without loss of generality, we consider $x(k,\Delta)\in \Real$ and let $\v{X}(k) = [x_0(k), x_1(k), \cdots, x_{p}(k)]^T$ be the gPC expansion of $x(k,\Delta)$. The $p^{th}$ moment in terms of $x_i(k)$ are then given by
\begin{equation}
\begin{split}
& m_p(k) = \sum_{i_1=0}^P \cdots \sum_{{i_p}=0}^P x_{i_1}(k) \cdots x_{i_p}(k) \int_{\mathcal{D}_\Delta} \phi_{i_1}(\Delta) \cdots \\
& \phi_{i_p}(\Delta)f(\Delta)d\Delta. 
\end{split}
\label{eqn:mp}
\end{equation}

Thus, minimizing $C_f(x(N),\Delta)$ in \eqnref{eqn:term-cost} minimizes all moments of $x(N,\Delta)$. Consequently, constraining the probability density function of $x(N,\Delta)$. 

\subsection{State and Control Constraints}
In this section we present the state and control constraints for the receding
horizon policy.
\subsubsection{Expectation Based Constraints}
Here we first consider constraints of the following form,
\begin{eqnarray}
\Expected{x(k,\Delta)^T H_{x} x(k,\Delta) + G_x x(k,\Delta)} & \le & \alpha_{i,x},\label{eqn:xMeanConstr}\\
\Expected{u(k,\Delta)^T H_{u} u(k,\Delta) + G_u u(k,\Delta)} & \le & \alpha_{i,u}, \label{eqn:uMeanConstr}
\end{eqnarray}
for $k=0\ldots N$.
These constraints are on the {\it expected value} of the quadratic functions.  Thus, instead of requiring
that the constraint be met for all trajectories, they instead imply that the constraints should be satisfied on average.
These constraints can be expressed in terms of the gPC states as 
\begin{eqnarray}
\v{X}(k)^T \bar{H}_{x} \v{X}(k) + \bar{G}_x \v{X}(k) & \le & \alpha_{i,x},\label{eqn:xpcMeanConstr}\\
\v{U}(k)^T \bar{H}_{u} \v{U}(k) + \bar{G}_u \v{U}(k) & \le & \alpha_{i,u},\label{eqn:upcMeanConstr}
\end{eqnarray}
where $\bar{H}_{x}=W \otimes H_{x}$, $\bar{H}_{u}=W \otimes H_{u}$, $\bar{G}_{x} =  G_{x}\left[I_n \; 0_{n\times np}\right]$, and $\bar{G}_{u} =  G_{u}\left[I_n \; 0_{n\times np}\right]$.

\subsubsection{Variance Based Constraints}
In many practical applications, it may be desirable to constrain the
second moment of the state trajectories, either at each time step or at final time. One means of achieving this
is to use a constraint of the form
\begin{equation}
\Trace{\Expected{(x(k)-\Expected{x(k)})(x(k)-\Expected{x(k)})^T}}\le \alpha_{\sigma^2}.
\end{equation}

For scalar $x$, the variance $\sigma^2(x)$ in terms of the gPC
expansions can be shown to be
\[
 \sigma^2  =  \Expected{x-\Expected{x}}^2 = \Expected{x^2} - \Expected{x}^2 = \mathbf{x}^T W \mathbf{x}- \Expected{x}^2,
  \]
where
\[
\begin{split}
& \Expected{x}  = \Expected{\sum_{i=0}^p x_i\phi_i}= \sum_{i=0}^p x_i \Expected{\phi_i} = \sum_{i=0}^p x_i \int_{\mathcal{D}_\Delta}\phi_i f d\Delta \\
& = \mathbf{x}^T F,
\end{split}
\]
and $F = \left[\begin{array}{cccc}1 \; 0 \; \cdots \;0 \end{array}\right]^{T}$. Therefore, $\sigma^2$ for scalar $x$ can be written in a compact form as
\begin{equation}
\sigma^2  = \mathbf{x}^T(W-FF^T)\mathbf{x}.
\end{equation}

In order to represent the covariance for $x\in\Real^{n}$, in terms of the gPC states, let us define 
$\Phi = [\phi_{0} \cdots \phi_{p+1}]^{T}$ and write $G = \int_{\mathcal{D}_{\Delta}}\Phi\Phi^{T}f d\Delta$. Let us represent a sub-vector of $\v{X}$, defined by elements $n_{1}$ to $n_{2}$, as $X_{n_{1}\cdots n{2}}$, where $n_{1}$ and $n_{2}$ are positive integers. Let us next define matrix  $M_{\v{X}}$, with subvectors of $\v{X}$, as $M_{\v{X}} = [\v{X}_{1\cdots n}\; \v{X}_{n+1 \cdots 2n} \; \cdots \v{X}_{np+1 \cdots n(p+1)}]$. For $x\in\Real^{n}$, it can be shown that 
\begin{equation}
\Expected{x} = (F\otimes I_{n})\v{X}, \label{eqn:mean}
\end{equation}
and the covariance can then be shown to be
\begin{equation}
\textbf{Cov}(x) = M_{\v{X}}GM_{\v{X}}^{T} - (F \otimes I_{n})\v{X}\v{X}^{T}(F^{T}\otimes I_{n}).
\label{eqn:cov}
\end{equation}

The trace of the covariance matrix $\textbf{Cov}(x)$ can then be written as
\[
\Trace{\textbf{Cov}(x)} = \v{X}^{T}((W-FF^{T})\otimes I_{n}) \v{X}.
\]
Therefore, a constraint of the type 
\[
\Trace{\textbf{Cov}(x(k))} \le \alpha_{\sigma^2}
\]
can be written in term of gPC states as
\begin{equation}
\v{X}^{T}Q_{\sigma^2} \v{X} \le \alpha_{\sigma^2},
\label{eqn:cov-constraint}
\end{equation}
where $Q_{\sigma^2}=(W-FF^T)\otimes I_{n}$. 

\section{Stability of the RHC Policy}
Here we show the stability properties of the receding horizon policy
when it is applied to the system in \eqnref{eq:pcrhcdyn}. Using gPC theory we can convert the underlying stochastic RHC formulation in $x(t,\Delta)$ and $u(t,\Delta)$ to a deterministic RHC formulation in $\v{X}(k)$ and $\v{U}(k)$. The stability of $\v{X}(k)$ in an RHC setting, with suitable terminal controller, can be proved using results by \cite{goodwin2005cca}, which shows that $\lim_{k\rightarrow \infty} \v{X}(k)\rightarrow 0$, when a receding horizon policy is employed. To relate this result to the stability of $x(k,\Delta)$, we first present the following known result in stochastic stability in terms of the moments of $x(k,\Delta)$. For stochastic dynamical systems in general, stability of moments is a weaker definition of stability than the \textit{almost sure stability} definition. However, the two definitions are equivalent for linear autonomous systems (pg. 296, \cite{Khasminski} also pg. 349 \cite{Chen}). Here we present the definition of asymptotic stability in the $p^{th}$ moment for discrete-time systems. 

\begin{definition}
The zero equilibrium state is said to be stable in the $p^{th}$ moment if $\forall \epsilon>0, \, \exists \delta > 0$ such that 
\begin{equation}
\sup_{k\geq0} \Expected{x(k,\Delta)^p} \leq \epsilon, \; \forall x(0,\Delta):||x(0,\Delta)||\leq\delta, \forall \Delta \in \mathcal{D}_{\Delta}.
\label{eqn:stab1}
\end{equation}
\end{definition}

\begin{definition}
The zero equilibrium state is said to be asymptotically stable in the $p^{th}$ moment if it is stable in $p^{th}$ moment and 
\begin{equation}
\lim_{k\rightarrow \infty} \Expected{x(k,\Delta)^p} = 0, \label{eqn:stab2}\end{equation}
for all $x(0,\Delta)$ in the neighbourhood of the zero equilibrium.
\end{definition}

\begin{proposition}
For the system in \eqnref{eqn:lti}, $\lim_{k\rightarrow \infty} \v{X}(k)\rightarrow 0$ is a sufficient condition for the asymptotic stability of the zero equilibrium state, in all moments.
\end{proposition}
\begin{proof}
To avoid tensor notation and without loss of generality, we consider $x(k,\Delta)\in \Real$ and let $\v{X}(k) = [x_0(k), x_1(k), \cdots, x_{p}(k)]^T$ be the gPC expansion of $x(k,\Delta)$. The moments in terms of $x_i(k)$ are given by eqn.(\ref{eqn:mp}).
Therefore, if $ \lim_{k\rightarrow \infty} \v{X}(k)\rightarrow 0$ $\implies \lim_{k\rightarrow \infty} x_i(k) \rightarrow 0$. Consequently, $\lim_{k\rightarrow \infty} m_i(k) \rightarrow 0$ for $i=1,2,\cdots$, and \eqnref{eqn:stab2} is satisfied.
This completes the proof. $\square$
\end{proof}

\section{Numerical Example}
Here we consider the following linear system, 
similar to that considered in \cite{primbs2009stochastic},  
\begin{equation}
x(k+1) = (A+G(\Delta))x(k)+Bu(k)
\end{equation}
where
\[
A=\left[\begin{array}{cc}1.02&-0.1\\.1&.98\end{array}\right],\, B=\left[
\begin{array}{c}0.1\\ 0.05\end{array}\right],\,
G = \left[\begin{array}{cc}0.04&0\\ 0&0.04\end{array}\right]\Delta.
\]
The system in consideration is open-loop unstable and the uncertainty appears
linearly in the $G$ matrix. Here, $\Delta \in [-1,1]$ and is governed
by a uniform distribution, that doesn't change with time. Consequently, Legendre polynomials is used for gPC approximation and polynomials up to $4^{th}$ order are used to formulate the control. Additionally, we assume that there is no uncertainty in the initial condition. The expectation based constraint is imposed on $x(k,\Delta))$ as
\[
\Expected{\;[1 \;\; 0 ] x(k,\Delta)\;}\ge -1,
\]
which in terms of the gPC states, this corresponds to
\[
\left[\begin{array}{cc}1&\v{0}_{1\times 2p+1}\end{array}\right]\v{X}(k)\ge -1.
\]

The terminal controller is designed using probabilistic LQR design techniques described by \cite{fisher2008sld}. The cost matrices used to determine the
terminal controller are
\[
Q = \left[\begin{array}{cc}2&0\\0&5\end{array}\right],\, R = 1.
\]

\Fig{fig:rhc} illustrates the performance of the proposed RHC policy The resulting optimization problem is a nonlinear programming problem which has been solved using MATLAB's \texttt{fmincon(...)} function. From the figure, we see that the constraint on the expected value of $x_1$ has been satisfied and the RHC algorithm was able to stabilize the system. These plots have been obtained using $4^{th}$ order gPC approximation of the stochastic dynamics.

\section{Summary}
In this paper we present a RHC strategy for linear discrete time systems with probabilistic system parameters. We have used the polynomial chaos framework to design stochastic RHC algorithms in an equivalent deterministic setting. The controller structure has an open loop component that controls the mean behavior of the system, and a state-feedback component that controls deviations about the mean trajectory. This controller structure results in a polynomial optimization problem with polynomial constraints that is solved in the general nonlinear programming framework. Theoretical guarantees for the stability of the proposed algorithm has also been presented. Performance of the RHC algorithm has been assessed using a two dimensional dynamical system. 

\begin{figure}[h!]
\includegraphics[width=0.45\textwidth]{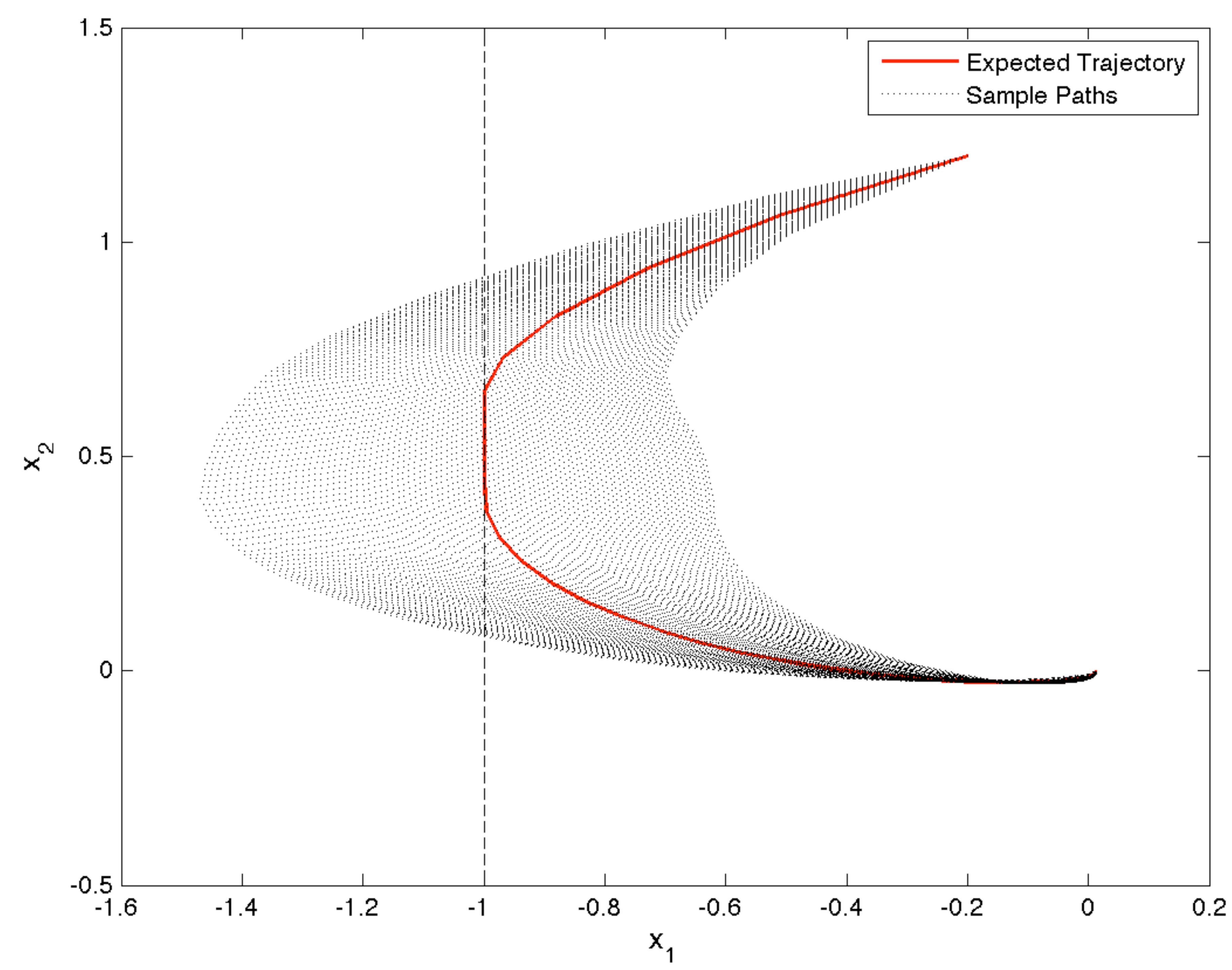}
\caption{State trajectories with expectation constraints.}\label{fig:rhc}
\end{figure}

\bibliographystyle{alpha}        
\bibliography{pc,rhc}

\begin{thebibliography}{37}
\providecommand{\natexlab}[1]{#1}
\providecommand{\url}[1]{\texttt{#1}}
\expandafter\ifx\csname urlstyle\endcsname\relax
  \providecommand{\doi}[1]{doi: #1}\else
  \providecommand{\doi}{doi: \begingroup \urlstyle{rm}\Url}\fi

\bibitem[Askey and Wilson(1985)]{Askey-Polynomials}
R.~Askey and J.~Wilson.
\newblock {Some basic hypergeometric polynomials that generalize jacobi
  polynomials}.
\newblock \emph{Memoirs Amer. Math. Soc.}, 319, 1985.

\bibitem[Batina et~al.(2002)Batina, Stoorvogel, and Weiland]{Batina}
I.~Batina, A.~A. Stoorvogel, and S.~Weiland.
\newblock {Optimal control of linear, stochastic systems with state and input
  constraints}.
\newblock \emph{Proceedings of the 41st IEEE Conference on Decision and
  Control}, 2:\penalty0 1564-- 1569, 2002.

\bibitem[Bemporad and Morari(1999)]{MPCSurvey}
A.~Bemporad and M.~Morari.
\newblock {Robust model predictive control: A survey}.
\newblock Technical report, Automatic Control Laboratory, Swiss Federal
  Institute of Technology (ETH), Physikstrasse 3, CH-8092 Z\"{u}rich,
  Switzerland, www.control.ethz.ch, 1999.

\bibitem[Bertsekas(2005)]{bertsekas2005dpa}
D.~Bertsekas.
\newblock {Dynamic programming and suboptimal control: A survey from ADP to
  MPC}.
\newblock \emph{European Journal of Control}, 11\penalty0 (4-5):\penalty0
  310--334, 2005.

\bibitem[Bhattacharya et~al.(2002)Bhattacharya, Balas, Kaya, and
  Packard]{raktimf16}
Raktim Bhattacharya, Gary~J. Balas, M.~Alpay Kaya, and Andy Packard.
\newblock Nonlinear receding horizon control of an f-16 aircraft.
\newblock \emph{Journal of Guidance, Control, and Dynamics}, 25\penalty0
  (5):\penalty0 924--931, 2002.

\bibitem[Bitmead et~al.(1990)Bitmead, Gevers, and Wertz]{rhcbook1}
R.R. Bitmead, M.~Gevers, and V.~Wertz.
\newblock \emph{Adaptive Optimal Control: The Thinking Man's GPC}.
\newblock International Series in Systems and Control Engineering. Prentice
  Hall, Englewood Cliffs, NJ, 1990.

\bibitem[Cameron and Martin(1947)]{CameronMartin}
R.~H. Cameron and W.~T. Martin.
\newblock {The orthogonal development of non-linear functionals in series of
  fourier-hermite functionals}.
\newblock \emph{The Annals of Mathematics}, 48\penalty0 (2):\penalty0 385--392,
  1947.

\bibitem[Cannon et~al.(2009)Cannon, Kouvaritakis, and Wu]{Cannon2009167}
M.~Cannon, B.~Kouvaritakis, and X.~Wu.
\newblock Model predictive control for systems with stochastic multiplicative
  uncertainty and probabilistic constraints.
\newblock \emph{Automatica}, 45\penalty0 (1):\penalty0 167 -- 172, 2009.

\bibitem[Chen and Hsu(1995)]{Chen}
By~G. Chen and S.~H. Hsu.
\newblock \emph{Linear Stochastic Control Systems}.
\newblock CRC Press, 1995.

\bibitem[Darlington et~al.(2000)Darlington, Pantelides, Rustem, and
  Tanyi]{darlington2000decreasing}
J.~Darlington, CC~Pantelides, B.~Rustem, and BA~Tanyi.
\newblock {Decreasing the sensitivity of open-loop optimal solutions in
  decision making under uncertainty}.
\newblock \emph{European Journal of Operational Research}, 121\penalty0
  (2):\penalty0 343--362, 2000.

\bibitem[de~la Penad et~al.(2005)de~la Penad, Bemporad, and
  Alamo]{delapenad2005spa}
DM~de~la Penad, A.~Bemporad, and T.~Alamo.
\newblock {Stochastic programming applied to model predictive control}.
\newblock In \emph{44th IEEE Conference on Decision and Control, 2005 and 2005
  European Control Conference. CDC-ECC'05}, pages 1361--1366, 2005.

\bibitem[Fisher and Bhattacharya(2008{\natexlab{a}})]{fisher2008sld}
J.~Fisher and R.~Bhattacharya.
\newblock {On stochastic LQR design and polynomial chaos}.
\newblock In \emph{American Control Conference, 2008}, pages 95--100,
  2008{\natexlab{a}}.

\bibitem[Fisher and Bhattacharya(2008{\natexlab{b}})]{pctrajgen}
J.~Fisher and R.~Bhattacharya.
\newblock {Optimal Trajectory Generation with Probabilistic System Uncertainty
  Using Polynomial Chaos}.
\newblock \emph{In Press Journal of Dynamic Systems, Measurement and Control},
  2008{\natexlab{b}}.

\bibitem[Fisher et~al.(2007)Fisher, Bhattacharya, and Vadali]{myrhc}
James Fisher, Raktim Bhattacharya, and S.~R. Vadali.
\newblock Spacecraft momentum management and attitude control using a receding
  horizon approach.
\newblock In \emph{Proceedings of the 2007 AIAA Guidance, Navigation, and
  Control Conference and Exhibit}, Hilton Head, SC, August 2007. AIAA.

\bibitem[Ghanem and Red-Horse(1999)]{pcSolids1}
Roger Ghanem and John Red-Horse.
\newblock {Propagation of probabilistic uncertainty in complex physical systems
  using a stochastic finite element approach}.
\newblock \emph{Phys. D}, 133\penalty0 (1-4):\penalty0 137--144, 1999.
\newblock ISSN 0167-2789.
\newblock \doi{http://dx.doi.org/10.1016/S0167-2789(99)00102-5}.

\bibitem[Ghanem and Spanos(1991)]{pcFEM}
Roger~G. Ghanem and Pol~D. Spanos.
\newblock \emph{Stochastic Finite Elements: A Spectral Approach}.
\newblock Springer-Verlag Inc., New York, NY, 1991.
\newblock ISBN 0-387-97456-3.

\bibitem[Goodwin et~al.(2005)Goodwin, Seron, and De~Dona]{goodwin2005cca}
G.C. Goodwin, M.~Seron, and J.~De~Dona.
\newblock \emph{{Constrained control and estimation: an optimisation
  approach}}.
\newblock Springer, 2005.

\bibitem[Grimm et~al.(2004)Grimm, Messina, Tuna, and Teel]{rhcnonrobust}
Gene Grimm, Michael~J. Messina, Sezai~E. Tuna, and Andrew~R. Teel.
\newblock Examples when nonlinear model predictive control is nonrobust.
\newblock \emph{Automatica}, 40:\penalty0 1729--1738, 2004.

\bibitem[Hou et~al.(2006)Hou, Luo, Rozovskii, and Zhou]{pcFluids2}
Thomas~Y. Hou, Wuan Luo, Boris Rozovskii, and Hao-Min Zhou.
\newblock {Wiener chaos expansions and numerical solutions of randomly forced
  equations of fluid mechanics}.
\newblock \emph{J. Comput. Phys.}, 216\penalty0 (2):\penalty0 687--706, 2006.
\newblock ISSN 0021-9991.
\newblock \doi{http://dx.doi.org/10.1016/j.jcp.2006.01.008}.

\bibitem[Jadbabaie et~al.(1999)Jadbabaie, Yu, and Hauser]{Ali_ACC99}
A.~Jadbabaie, J.~Yu, and J.~Hauser.
\newblock Stabilizing receding horizon control of nonlinear systems: A control
  lyapunov function approach.
\newblock \emph{Proceedings of the 1999 American Control Conference},
  3:\penalty0 1535--1539, 1999.

\bibitem[Khas'minskii(1969)]{Khasminski}
R.~Z. Khas'minskii.
\newblock \emph{Stability of Systems of Differential Equations in the Presence
  of Random Disturbances (in Russian)}.
\newblock Nauka, Moscow, 1969.

\bibitem[Kouvaritakis et~al.(2000)Kouvaritakis, Rossiter, and
  Schuurmans]{rhcefficient}
B.~Kouvaritakis, J.~A. Rossiter, and J.~Schuurmans.
\newblock Efficient robust predictive control.
\newblock \emph{IEEE Transactions on Automatic Control}, 45\penalty0
  (8):\penalty0 1545--1549, 2000.

\bibitem[Kwon(1994)]{kwon94}
W.H. Kwon.
\newblock Advances in predictive control: Theory and application.
\newblock \emph{1st Asian Control Conference, Tokyo}, 1994.

\bibitem[Lee and Yu(1997)]{rhcworstcase}
J.~H. Lee and Zhenghong Yu.
\newblock Worst-case formulations of model predictive control for systems with
  bounbded parameters.
\newblock \emph{Automatica}, 33\penalty0 (5):\penalty0 763--781, 1997.

\bibitem[Lee and Cooley(1998)]{lee1998ofc}
JH~Lee and BL~Cooley.
\newblock {Optimal feedback control strategies for state-space systems with
  stochastic parameters}.
\newblock \emph{IEEE Transactions on Automatic Control}, 43\penalty0
  (10):\penalty0 1469--1475, 1998.

\bibitem[Mayne et~al.(2000{\natexlab{a}})Mayne, Rawlings, Rao, and
  Scokaert]{NLRHC_survey}
D.~Mayne, J.~Rawlings, C.~Rao, and P.~Scokaert.
\newblock Constrained model predictive control, stability and optimality.
\newblock \emph{Automatica}, 36:\penalty0 789--814, 2000{\natexlab{a}}.

\bibitem[Mayne et~al.(2000{\natexlab{b}})Mayne, Rawlings, Rao, and
  Scokaert]{mayne2000cmp}
D.Q. Mayne, J.B. Rawlings, C.V. Rao, and PO~Scokaert.
\newblock {Constrained model predictive control: Stability and optimality}.
\newblock \emph{AUTOMATICA-OXFORD-}, 36:\penalty0 789--814, 2000{\natexlab{b}}.

\bibitem[Nagy and Braatz(2003)]{nagy2003robust}
Z.K. Nagy and R.D. Braatz.
\newblock {Robust nonlinear model predictive control of batch processes}.
\newblock \emph{AIChE Journal}, 49\penalty0 (7):\penalty0 1776--1786, 2003.

\bibitem[Prabhakar et~al.(2008)Prabhakar, Fisher, and Bhattacharya]{vinh-JGCD}
A.~Prabhakar, J.~Fisher, and R.~Bhattacharya.
\newblock {Polynomial Chaos Based Analysis of Probabilistic Uncertainty in
  Hypersonic Flight Dynamics}.
\newblock \emph{submitted AIAA Journal of Guidance, Control, and Dynamics},
  2008.

\bibitem[Primbs(1999)]{Primbs-Thesis}
J.A. Primbs.
\newblock \emph{Nonlinear Optimal Control: A Receding Horizon Approach}.
\newblock PhD thesis, California Institute of Technology, Pasadena, CA, 1999.

\bibitem[Primbs and Sung(2009)]{primbs2009stochastic}
J.A. Primbs and C.H. Sung.
\newblock Stochastic receding horizon control of constrained linear systems
  with state and control multiplicative noise.
\newblock \emph{Automatic Control, IEEE Transactions on}, 54\penalty0
  (2):\penalty0 221--230, 2009.

\bibitem[Qin and Badgwell(1996)]{MPC_Overview}
S.J. Qin and T.~Badgwell.
\newblock An overview of industrial model predictive control technology.
\newblock \emph{AIChE Symposium Series}, 93:\penalty0 232--256, 1996.

\bibitem[Rakovi\'{c} et~al.(2006)Rakovi\'{c}, Teel, Mayne, and
  Astolfi]{rakovic}
S.~V. Rakovi\'{c}, A.~R. Teel, D.~Q. Mayne, and A.~Astolfi.
\newblock Simple robust control invariant tubes for some classes of nonlinear
  discrete time systems.
\newblock In \emph{Proceedings of the 45th IEEE Conference on Decision and
  Control}, pages 6397--6402, San Diego, CA, December 2006. IEEE.

\bibitem[Schetzen(2006)]{volterra}
M.~Schetzen.
\newblock \emph{The Volterra and Wiener Theories of Nonlinear Systems}.
\newblock Krieger Pub., 2006.

\bibitem[van Hessem and Bosgra(2002)]{hessem1}
D.~H. van Hessem and O.~H. Bosgra.
\newblock A conic reformulation of model predictive control including bounded
  and stochastic disturbances and input constraints.
\newblock In \emph{Proceedings of the 2002 Conference on Decision and Control},
  volume~4, pages 4643--4648, Las Vegas, NV, December 2002. IEEE.

\bibitem[Wiener(1938)]{wienerPC}
N.~Wiener.
\newblock {The homogeneous chaos}.
\newblock \emph{American Journal of Mathematics}, 60\penalty0 (4):\penalty0
  897--936, 1938.

\bibitem[Xiu and Karniadakis(2002)]{Xiu}
Dongbin Xiu and George~Em Karniadakis.
\newblock {The wiener--askey polynomial chaos for stochastic differential
  equations}.
\newblock \emph{SIAM J. Sci. Comput.}, 24\penalty0 (2):\penalty0 619--644,
  2002.
\newblock ISSN 1064-8275.

\end{thebibliography}

\end{document}